\documentclass{amsart}
\usepackage{setspace, amsmath, amsthm, amssymb, amsfonts, amscd, epic, graphicx, ulem, dsfont}
\usepackage[T1]{fontenc}
\usepackage{multirow}
\usepackage{bbm}

\newtheorem{thm}{Theorem}[section]
\newtheorem{lem}{Lemma}[section]
\newtheorem{pro}{Proposition}[section]

\theoremstyle{abstract}

\theoremstyle{remark}
\newtheorem*{rema}{Remark}

\newtheorem{exa}{\textbf{Example}}

\theoremstyle{definition}

\newcommand{\R}{\mathbb{R}}

\newcommand{\N}{\mathbb{N}}

\makeatletter \@namedef{subjclassname@2010}{
  \textup{2010} Mathematics Subject Classification}
\makeatother

\begin{document}

\title[On Kaplansky Theorem]{Generalizations of Kaplansky Theorem Related to Linear Operators}
\author[Benali and Mortad]{Abdelkader Benali and Mohammed Hichem Mortad $^*$}

\address{Department of
Mathematics, University of Oran, B.P. 1524, El Menouar, Oran 31000.
Algeria.\newline {\bf Mailing address}:
\newline Dr Mohammed Hichem Mortad \newline BP 7085
Es-Seddikia\newline Oran
\newline 31013 \newline Algeria}

\email{nasro-91@hotmail.fr}

\email{mhmortad@gmail.com, mortad@univ-oran.dz.}

\begin{abstract}
The purpose of this paper is to generalize a very famous result on
products of normal operators, due to I. Kaplansky. The context of
generalization is that of bounded hyponormal and unbounded normal
operators on complex separable Hilbert spaces. Some examples "spice
up" the paper.
\end{abstract}

\subjclass[2010]{Primary 47B20; Secondary 47A05}

\keywords{Products of operators. Bounded and unbounded: Normal,
hyponormal, subnormal operators. Kaplansky theorem. Fuglede-Putnam
theorem.}

\thanks{$*$ Corresponding author.}

\maketitle

\section{Introduction}

Normal operators are a major class of bounded and unbounded
operators. Among their virtues, they are the largest class of single
operators for which the spectral theorem is proved (cf.
\cite{SCHMUDG-book-2012}). There are other classes of interesting
non-normal operators such as hyponormal and subnormal operators
(among others). They have been of interest to many mathematicians
and have been extensively investigated enough so that even
monographs have been devoted to them. See for instance
\cite{CON-subnormal book} and \cite{Martin-Putinar-hyponormal
operators book}.

In this paper we are mainly interested in generalizing the following
result to unbounded normal and bounded hyponormal operators:

\begin{thm}[Kaplansky, \cite{Kapl}]\label{Kapl-bounded}
Let $A$ and $B$ be two bounded operators on a Hilbert space such
that $AB$ and $A$ are normal. Then $B$ commutes with $AA^*$ iff $BA$
is normal.
\end{thm}

Before recalling some essential background, we make the following
observation:

\textit{All operators are linear and are defined on a separable
complex Hilbert space, which we will denote henceforth by $H$.}

A bounded operator $A$ on $H$ is said to be normal if $AA^*=A^*A$.
$A$ is called hyponormal if $AA^*\leq A^*A$, that is iff
$\|A^*x\|\leq\|Ax\|$ for all $x\in H$. Hence a normal operator is
always hyponormal. Obviously, a hyponormal operator need not be
normal. However, and in a finite-dimensional setting, a hyponormal
operator is normal too. This is proved via a nice and simple trace
argument (see e.g. \cite{halmos-book-1982}).

Since the paper is also concerned with unbounded operators, and for
the readers convenience, we recall some known notions and results
about unbounded operators.

If $A$ and $B$ are two unbounded operators with domains $D(A)$ and
$D(B)$ respectively, then $B$ is said to be an extension of $A$, and
we denote it by $A\subset B$, if $D(A)\subset D(B)$ and $A$ and $B$
coincide on each element of $D(A)$. An operator $A$ is said to be
densely defined if $D(A)$ is dense in $H$. The (Hilbert) adjoint of
$A$ is denoted by $A^*$ and it is known to be unique if $A$ is
densely defined. An operator $A$ is said to be closed if its graph
is closed in $H\times H$. We say that the unbounded $A$ is
self-adjoint if $A=A^*$, and we say that it is normal if $A$ is
closed and $AA^*=A^*A$. Recall also that the product $BA$ is closed
if for instance $B$ is closed and $A$ is bounded, and that if $A$,
$B$ and $AB$ are densely-defined, then only $B^*A^*\subset (AB)^*$
holds; and if further $A$ is assumed to be bounded, then
$B^*A^*=(AB)^*$.

The notion of hyponormality extends naturally to unbounded
operators. an unbounded $A$ is called hyponormal if:
\begin{enumerate}
  \item $D(A)\subset D(A^*)$,
  \item $\|A^*x\|\leq\|Ax\|$ for all $x\in D(A)$.
\end{enumerate}

It is also convenient to recall the following theorem which appeared
in \cite{STO}, but we state it in the form we need.
\begin{thm}[Stochel]\label{STOCHEL-ASYMMETRIC}
If $T$ is a closed subnormal (resp. closed hyponormal) operator and
$S$ is a closed hyponormal (resp. closed subnormal) operator
verifying $XT^*\subset SX$ where $X$ is a bounded operator, then
both $S$ and $T^*$ are normal once $\ker X=\ker X^*=\{0\}$.
\end{thm}

Any other result or notion (such as the classical Fuglede-Putnam
theorem, the polar decomposition, subnormality etc...) will be
assumed to be known by readers. For more details, the interested
reader is referred to \cite{Bereb_book}, \cite{CON}, \cite{GGK},
\cite{RUD} and \cite{SCHMUDG-book-2012}. For other works related to
products of normal (bounded and unbounded) operators, the reader may
consult \cite{Gheondea}, \cite{kitt-prod-norm},
\cite{Mortad-Demmath}, \cite{mortad-commutatvity-devinatz-2013} and
\cite{patel-Ramanujan-1981-sum-product}, and the references therein.

\section{Main Results: The Bounded Case}

The following known lemma is essential (we include a proof):

\begin{lem}\label{Lemma A>B UAU*>UBU*}
Let $S$ and $T$ be two bounded self-adjoint operators on a Hilbert
space $H$. If $U$ is any operator, then
\[S\geq T \Longrightarrow USU^*\geq UTU^*.\]
\end{lem}

\begin{proof}
Let $x\in H$. We have
\begin{align*}
<USU^*x,x>= & <SU^*x,U^*x>\\
\geq & <TU^*x,U^*x>\\
=&<UTU^*x,x>.
\end{align*}
\end{proof}

As a direct application of the previous result we have the following
Kaplansky-like theorem:

\begin{pro}\label{normal A hypo Ab P2B=BP2 implying BA hypo
Proposition} Let $A$ and $B$ be two bounded operators on a Hilbert
space such that $A$ is normal and $AB$ is hyponormal. Then
\[AA^*B=BAA^*\Longrightarrow\text{ $BA$ is hyponormal.}\]
\end{pro}

\begin{proof} Since $A$ is normal, we know that
\[A=PU=UP\]
where $P$ is positive and $U$ is unitary. Hence
\[AA^*B=BAA^*\Longrightarrow P^2B=BP^2\Longrightarrow PB=BB\]
so that
\[U^*ABU=U^*UPBU=PBU=BA.\]

Finally, we have
\begin{align*}
BA(BA)^*&=(U^*(AB)U)(U^*ABU)^*\\
&=U^*ABUU^*(AB)^*U\\
&=U^*(AB)(AB)^*U\\
&\leq U^*(AB)^*ABU\\
&=(BA)^*BA.
\end{align*}
\end{proof}

The reverse implication does not hold in the previous result (even
if $A$ is self-adjoint) as shown in the following example:

\begin{exa}
Let $A$ and $B$ be acting on the standard basis $(e_n)$ of
$\ell^2(\N)$ by:
\[Ae_n=\alpha_ne_n \text{ and } Be_n=e_{n+1},~\forall n\geq 1\]
respectively. Assume further that $\alpha_n$ is bounded,
\textit{real-valued} and \textit{positive}, for all $n$. Hence $A$
is self-adjoint (hence normal!) and positive. Then
\[ABe_n=\alpha_{n}e_{n+1},~\forall n\geq 1.\]

For convenience, let us carry out the calculations as infinite
matrices. Then

\[ AB=\begin{bmatrix} 0 & 0 &  &  & &\text{\Large{0}} \\ \alpha_1 & 0 &
0 &  &
\\ 0 & \alpha_2 & 0 & 0 &
 \\
 & 0 & \alpha_3 & 0 & \ddots \\
 &  & 0 & \ddots & 0 & \\
\text{\Large{0}} &  &  & \ddots & \ddots & \ddots
\end{bmatrix} \text{ so that } (AB)^*=\begin{bmatrix} 0 & \alpha_1 &  &  & &\text{\Large{0}} \\ 0 & 0 &
\alpha_2 &  &
\\ 0 & 0 & 0 & \alpha_3 &
 \\
 & 0 & 0 & 0 & \ddots \\
 &  & 0 & \ddots & 0 & \\
\text{\Large{0}} &  &  & \ddots & \ddots & \ddots
\end{bmatrix}.\]
Hence
\[ AB(AB)^*=\begin{bmatrix} 0 & 0 &  &  & &\text{\Large{0}} \\ 0 &\alpha_1^2 &
0 &  &
\\ 0 & 0 & \alpha_2^2 & 0 &
 \\
 & 0 & 0 & \alpha_3^2 & \ddots \\
 &  & 0 & \ddots & \ddots & \\
\text{\Large{0}} &  &  & \ddots & \ddots & \ddots
\end{bmatrix}\] \text{ and }
\[(AB)^*AB=\begin{bmatrix} \alpha_1^2 & 0 &  &  & &\text{\Large{0}} \\ 0
&\alpha_2^2 & 0 &  &
\\ 0 & 0 & \alpha_3^2 & 0 &
 \\
 & 0 & 0 & \ddots & \ddots \\
 &  & 0 & \ddots & \ddots & \\
\text{\Large{0}} &  &  & \ddots & \ddots & \ddots
\end{bmatrix}.\]
It thus becomes clear that $AB$ is hyponormal iff $\alpha_n\leq
\alpha_{n+1}$.

Similarly
\[BAe_n=\alpha_{n+1}e_{n+1},~\forall n\geq 1.\]
Whence the matrix representing $BA$ is given by:
\[BA=\begin{bmatrix} 0 & 0 &  &  & &\text{\Large{0}} \\ \alpha_2 & 0 &
0 &  &
\\ 0 & \alpha_3 & 0 & 0 &
 \\
 & 0 & \alpha_4 & 0 & \ddots \\
 &  & 0 & \ddots & 0 & \\
\text{\Large{0}} &  &  & \ddots & \ddots & \ddots
\end{bmatrix} \text{ so that } (BA)^*=\begin{bmatrix} 0 & \alpha_2 &  &  & &\text{\Large{0}} \\ 0 & 0 &
\alpha_3 &  &
\\ 0 & 0 & 0 & \alpha_4 &
 \\
 & 0 & 0 & 0 & \ddots \\
 &  & 0 & \ddots & 0 & \\
\text{\Large{0}} &  &  & \ddots & \ddots & \ddots
\end{bmatrix}.\]
Therefore,
\[BA(BA)^*=\begin{bmatrix} 0 & 0 &  &  & &\text{\Large{0}} \\ 0 &\alpha_2^2 &
0 &  &
\\ 0 & 0 & \alpha_3^2 & 0 &
 \\
 & 0 & 0 & \alpha_4^2 & \ddots \\
 &  & 0 & \ddots & \ddots & \\
\text{\Large{0}} &  &  & \ddots & \ddots & \ddots
\end{bmatrix}\]
and
\[(BA)^*BA=\begin{bmatrix} \alpha_2^2 & 0 &  &  & &\text{\Large{0}} \\ 0
&\alpha_3^2 & 0 &  &
\\ 0 & 0 & \alpha_4^2 & 0 &
 \\
 & 0 & 0 & \ddots & \ddots \\
 &  & 0 & \ddots & \ddots & \\
\text{\Large{0}} &  &  & \ddots & \ddots & \ddots
\end{bmatrix}.\]
Accordingly, $BA$ is hyponormal iff $\alpha_n\leq \alpha_{n+1}$
(thankfully, this is the same condition for the hyponormality of
$AB$).

Finally,

\[BA^2=\begin{bmatrix} 0 & 0 &  &  & &\text{\Large{0}} \\ \alpha_1^2 & 0 &
0 &  &
\\ 0 & \alpha_2^2 & 0 & 0 &
 \\
 & 0 & \alpha_3^2 & 0 & \ddots \\
 &  & 0 & \ddots & 0 & \\
\text{\Large{0}} &  &  & \ddots & \ddots & \ddots
\end{bmatrix}\neq A^2B=\begin{bmatrix} 0 & 0 &  &  & &\text{\Large{0}} \\ \alpha_2^2 & 0 &
0 &  &
\\ 0 & \alpha_3^2 & 0 & 0 &
 \\
 & 0 & \alpha_4^2 & 0 & \ddots \\
 &  & 0 & \ddots & 0 & \\
\text{\Large{0}} &  &  & \ddots & \ddots & \ddots
\end{bmatrix}\]

\begin{rema}
An explicit example of such an $(\alpha_n)$ verifying the required
hypotheses would be to take:
\[\left\{\begin{array}{c}
           \alpha_1=0 \\
           \alpha_{n+1}=\sqrt{2+\alpha_n}
         \end{array}
\right.\] Then $(\alpha_n)$ is bounded (in fact, $0\leq \alpha_n<2$,
for all $n$), increasing and such that $\alpha_1=0\neq
\alpha_2=\sqrt2$.
\end{rema}
\end{exa}

Going back to Proposition \ref{normal A hypo Ab P2B=BP2 implying BA
hypo Proposition}, we observe that the result obviously holds by
replacing "hyponormal" by "co-hyponormal". Thus we have

\begin{pro}Let $A$ and $B$ be two bounded operators on a Hilbert
space such that $A$ is normal and $AB$ is co-hyponormal. Then
\[AA^*B=BAA^*\Longrightarrow\text{ $BA$ is co-hyponormal.}\]
\end{pro}

\begin{rema}
The same previous example, mutatis mutandis, works as a
counterexample to show that "$BA$ co-hyponormal $\Rightarrow
AA^*B=BAA^*$" need not hold.
\end{rema}

We now come to a very important result of the paper:

\begin{thm}\label{main theorem hyponormal cohyponormal}
Let $A,B$ be two bounded operators such that $A$ is also normal.
Assume that $AB$ is hyponormal and that $BA$ is co-hyponormal. Then
\[AA^*B=BA^*A \Longleftrightarrow BA \text{ and } AB \text{ are normal. }\]
\end{thm}

\begin{proof}\hfill
\begin{enumerate}
  \item "$\Longrightarrow$": Since $A$ is normal, we have $A=UP=PU$
  where $U$ is unitary and $P$ is positive. Since $AA^*B=BA^*A$, we
  obtain $P^2B=BP^2$  or just $BP=PB$ by the positivity of $P$.

  Therefore, we may write
  \begin{align*}
  (AB)^*AB=&(U(BA)U^*)^*(U(BA)U^*)\\
  =&U(BA)^*U^*U(BA)U^*\\
  =&U(BA)^*(BA)U^*\\
  \leq &U(BA)(BA)^*U^* \text{ (since $BA$ is co-hyponormal and by Lemma \ref{Lemma A>B UAU*>UBU*})}\\
  =&U(BA)U^*U(BA)^*U^*\\
  =&(AB)(AB)^*,
  \end{align*}
  that is $AB$ is co-hyponormal. Since it is already hyponormal, it
  immediately follows that $AB$ is normal.

  To prove that $BA$ is normal we apply a similar idea and we have
\begin{align*}
  (BA)(BA)^*=&(U^*(AB)U)(U^*(AB)U)^*\\
  =&U^*(AB)^*UU^*(AB)^*U\\
  =&U^*(AB)(AB)^*U\\
  \leq &U^*(AB)^*(AB)U \text{ (since $AB$ is hyponormal and by Lemma \ref{Lemma A>B UAU*>UBU*})}\\
  =&U^*(AB)^*UU^*(AB)U\\
  =&(BA)^*(BA),
  \end{align*}
  that is $BA$ is hyponormal, and since it is also co-hyponormal, we conclude that $BA$ is normal.
  \item "$\Longleftarrow$": To prove the the reverse implication, we
  use the celebrated Fuglede-Putnam theorem (see e.g. \cite{CON}) and we have:
  \begin{align*}
  ABA=ABA&\Longrightarrow A(BA)=(AB)A\\
  &\Longrightarrow A(BA)^*=(AB)^*A\\
  &\Longrightarrow AA^*B^*=B^*A^*A\\
  &\Longrightarrow BAA^*=A^*AB.
  \end{align*}
  This completes the proof.
\end{enumerate}
\end{proof}

\section{Main Results: The Unbounded Case}

We start this section by giving a counterexample that shows that the
same assumptions, as in Theorem \ref{Kapl-bounded}, would not yield
the same results if $B$ is an unbounded operator, let alone the case
where both $A$ and $B$ are unbounded.

What we want is a normal bounded operator $A$ and an unbounded (and
closed) operator $B$ such that $BA$ is normal, $A^*AB\subset BA^*A$
but $AB$ is not normal.
\begin{exa}\label{example counterexample main}
Let
\[Bf(x)=e^{x^2}f(x) \text{ and } Af(x)=e^{-x^2}f(x)\]
on their respective domains
\[D(B)=\{f\in L^2(\R):~e^{x^2}f\in L^2(\R)\} \text{ and } D(A)=L^2(\R).\]
Then it is well known that $A$ is bounded and self-adjoint (hence
normal), and that $B$ is self-adjoint (hence closed).

Now $AB$ is not normal for it is not closed as $AB\subset I$. $BA$
is normal as $BA=I$ (on $L^2(\R)$). Hence $AB\subset BA$ which
implies that
\[AAB\subset ABA\Longrightarrow AAB\subset ABA\subset BAA.\]
\end{exa}

Now, we state and prove the generalization of Theorem
\ref{Kapl-bounded} to unbounded operators. We have

\begin{thm}
Let $B$ be an unbounded closed operator and $A$ a bounded one such
that $AB$ and $A$ are normal.  Then
\[BA \text{ normal } \Longrightarrow A^*AB\subset BA^*A.\]
If further $BA$ is hyponormal (resp. subnormal), then

\[BA \text{ normal} \Longleftarrow A^*AB\subset BA^*A.\]
\end{thm}

\begin{proof}\hfill
\begin{enumerate}
  \item "$\Longrightarrow$": Since $AB$ and $BA$ are normal, the
  equation
  \[A(BA)=(AB)A\]
  implies that
  \[A(BA)^*=(AB)^*A\]
  by the Fuglede-Putnam theorem (see e.g. \cite{CON}). Hence
  \[AA^*B^*\subset B^*A^*A \text{ or } A^*AB\subset BA^*A.\]
  \item "$\Longleftarrow$": The idea of proof in this case is similar in core to Kaplansky's (cf. \cite{Kapl}). Let $A=UP$ be the polar decomposition of
  $A$, where $U$ is unitary and $P$ is positive (remember that they also
  commute and that $P=\sqrt{A^*A}$), then one may write
  \[U^*ABU=U^*UPBU=PBU\subset BP U=BA\]
  or
  \[U^*AB=U^*\overline{AB}=U^*((AB)^*)^*\subset BA U^*\]
  (by the closedness of $AB$). Since $(AB)^*$ is normal, it is
  closed and subnormal. Since $B$ is closed and $A$ is bounded, $BA$
  is closed. Since it is hyponormal, Theorem \ref{STOCHEL-ASYMMETRIC} applies and
  yields the normality of $BA$ as $U$ is invertible.

  The proof is very much alike in the case of subnormality.
\end{enumerate}
\end{proof}

Imposing another commutativity condition allows us to generalize
Theorem \ref{Kapl-bounded} to unbounded normal operators by
bypassing hyponormality and subnormality:

\begin{thm}\label{kaplansky all unbounded just normals}
Let $B$ be an unbounded closed operator and $A$ a bounded one such
that all of $AB$, $A$ and $B$ are all normal.  Then
\[A^*AB\subset BA^*A \text{ and } AB^*B\subset
B^*BA \Longrightarrow BA \text{ normal } .\]
\end{thm}

The proof is partly based on the following interesting result of
maximality of self-adjoint operators:

\begin{pro}[Devinatz-Nussbaum-von Neumann, \cite{DevNussbaum-von-Neumann}]\label{Devinatz-Nussbaum-von Neumann: T=T1T2}
Let $A$, $B$ and $C$ be unbounded self-adjoint operators. Then
\[A\subseteq BC \Longrightarrow A=BC.\]
\end{pro}

Now we give the proof of Theorem \ref{kaplansky all unbounded just
normals}:

\begin{proof}First, $BA$ is closed as $A$ is bounded and $B$ is
closed. So $BA(BA)^*$ (and $(BA)^*BA$) is self-adjoint. Then we have
\[A^*ABB^*\subset BA^*AB^*=BAA^*B^*\subset BA(BA)^*\]
and hence
\[BA(BA)^*\subset (A^*ABB^*)^*=BB^*A^*A\]
so that Proposition \ref{Devinatz-Nussbaum-von Neumann: T=T1T2}
gives us
\[BA(BA)^*=BB^*A^*A\]
for both $BB^*$ and $A^*A$ are self-adjoint since $B$ is closed and
$A$ is bounded respectively.

 Similarly
\[A^*AB^*B\subset A^*B^*BA\subset (BA)^*BA.\]
Adjointing the previous "inclusion" and applying again Proposition
\ref{Devinatz-Nussbaum-von Neumann: T=T1T2} yield
\[(BA)^*BA=B^*BA^*A=BB^*A^*A,\]
establishing the normality of $BA$.
\end{proof}

\bibliographystyle{amsplain}

\end{document}